\documentclass[12pt,reqno,fleqn]{amsart}  
\usepackage{tikz}
\usepackage{amsmath,amstext,amsthm,amssymb,amsxtra}
\usepackage[top=1.5in, bottom=1.5in, left=1.25in, right=1.25in]	{geometry}
\usepackage{lmodern}

\usepackage{graphics}

\usepackage{mathtools}
\mathtoolsset{showonlyrefs,showmanualtags}

\usepackage{hyperref} 
\hypersetup{
	colorlinks=true,       
	linkcolor=blue,          
	citecolor=magenta,        
	filecolor=magenta,      
	urlcolor=cyan           
}

\usepackage[msc-links]{amsrefs}

\newtheorem{theorem}{Theorem}[section]

\newtheorem{lemma}{Lemma}[section]
\newtheorem{corollary}{Corollary}[section]
\newtheorem{OldTheorem}{Theorem}
\theoremstyle{definition}
\newtheorem{definition}{Definition}[section]
\theoremstyle{definition}

\theoremstyle{remark}
\newtheorem{remark}{Remark}[section]

\theoremstyle{remark}

\def\sp{{\rm span}}

\def\Log{{\rm Log}}

\def\rank{{\rm rank\,}}

\def\ZU{\ensuremath{\mathfrak U}}
\def\ZV{\ensuremath{\mathfrak V}}

\def\MM^d{\ensuremath{\mathfrak M}}
\def\MM{\ensuremath{\mathcal M}}

\def\ZB{\ensuremath{\mathcal B}}

\def\ZZ{\ensuremath{\mathbb Z}}

\def\ZI{\ensuremath{\textbf 1}}
\def\x{\ensuremath{\textbf x}}
\def\n{\ensuremath{\textbf n}}
\def\k{\ensuremath{\textbf k}}
\def\m{\ensuremath{\textbf m}}
\def\u{\ensuremath{\textbf u}}
\def\t{\ensuremath{\textbf t}}

\def\v{\ensuremath{\textbf v}}
\def\s{\ensuremath{\textbf s}}
\def\zr{\ensuremath{\textbf r}}
\def\ZN{\ensuremath{\mathbb N}}

\def\ZP{\ensuremath{\mathcal P}}
\def\ZD{\ensuremath{\cal D}}

\def\ZR{\ensuremath{\mathbb R}}

\def\ZT{\ensuremath{\mathbb T}}

\def\ZD{\ensuremath{\mathcal D}}

\numberwithin{equation}{section}
\def\md#1#2\emd{\ifx0#1
	\begin{equation*} #2 \end{equation*}\fi  
	\ifx1#1\begin{equation}#2\end{equation}\fi   
	\ifx2#1\begin{align*}#2\end{align*}\fi   
	\ifx3#1\begin{align}#2\end{align}\fi    
	\ifx4#1\begin{gather*}#2\end{gather*}\fi  
	\ifx5#1\begin{gather}#2\end{gather}\fi   
	\ifx6#1\begin{multline*}#2\end{multline*}\fi  
	\ifx7#1\begin{multline}#2\end{multline}\fi  
	\ifx8#1\begin{multline*}\begin{split}#2\end{split}\end{multline*}\fi
	\ifx9#1\begin{multline}\begin{split}#2\end{split}\end{multline}\fi
}
\newcommand {\e }[1]{\eqref{#1}}
\newcommand {\lem }[1]{Lemma \ref{#1}}

\newcommand {\cor }[1]{Corollary \ref{#1}}

\newcommand {\trm }[1]{Theorem \ref{#1}}

\begin{document}
\title[] {On the convergence of multiple ergodic means}
\author{Grigori A. Karagulyan}
\address{Faculty of Mathematics and Mechanics, Yerevan State
	University, Alex Manoogian, 1, 0025, Yerevan, Armenia} 
\email{g.karagulyan@ysu.am}

\author{Michael T. Lacey}   
\address{ School of Mathematics, Georgia Institute of Technology, Atlanta GA 30332, USA}
\email {lacey@math.gatech.edu}
\thanks{Research supported in part by grant  from the US National Science Foundation, DMS-1949206}

\author{Vahan A. Martirosyan}
\address{Faculty of Mathematics and Mechanics, Yerevan State
	University, Alex Manoogian, 1, 0025, Yerevan, Armenia} 
\email{vahanmartirosyan2000@gmail.com}
\thanks{The work was supported by the Science Committee of RA, in the frames of the research project 21AG‐1A045 }

\subjclass[2010]{37A30, 37A46, 42B25}
\keywords{Ergodic theorems, strong maximal function, multiple ergodic sums}

\begin{abstract}
	Given sequence of measure preserving transformations $\ZU=\{U_k:\,k=1,2,\ldots, n\}$ on a measurable space $(X,\mu)$. We prove a.e. convergence of the ergodic means	
	\begin{equation}\label{1}
		\frac{1}{s_1\cdots s_{n}}\sum_{j_1=0}^{s_1-1}\cdots\sum_{j_n=0}^{s_n-1}f\left(U_1^{j_1}\cdots U_n^{j_n} x	\right)
	\end{equation}
	as $\min_j s_j\to\infty $, for any function $f\in L\log^{d-1}(X)$, where $d\le n$ is the rank of the transformations $\ZU$. The result gives a generalization of a theorem by N. Dunford and A. Zygmund, claiming the convergence of \e{1} in a narrower class of functions $L\log^{n-1}(X)$.
\end{abstract}

\maketitle  

\section{Introduction}

Birkhoff’s ergodic theorem is one of the most important and beautiful result of the probability
theory. The study of ergodic theorems started in 1931 by von Neumann and Birkhoff, having its origins
in statistical mechanics. Recall the definition of the measure-preserving transformation (see \cite{EiWa}).
\begin{definition}
	Let $(X,\ZB,\mu)$ be a probability space. A mapping $T: X\to X$ is said to be a measure-preserving transformation if for any measurable set $E\in \ZB$ the set $T^{-1}(E)$ is also measurable and $\mu(E)=\mu(T^{-1}(E))$. The combination $(X,\ZB,\mu,T)$ is called a measure-preserving system.
\end{definition}

\begin{OldTheorem}[Birkhoff]
	If $(X,\ZB,\mu, T)$ is a measure-preserving system, then for any function $f\in L^1(X)$ the averages
	\begin{equation}\label{a33}
		\frac{1}{n}\sum_{j=0}^{n-1}f(T^jx)
	\end{equation}
converge almost everywhere to a $T$-invariant function $\bar f$ as $n\to\infty$. 
\end{OldTheorem} 
There are different proofs and various generalizations of this classical theorem. Some of those clearly demonstrate strong link between the Lebesgue differentiation theory on $\ZR^n$ and  pointwise convergence of different type of ergodic averages. The following multiple version of Birkhoff's theorem, proved by Zygmund \cite{Zyg} and Dunford \cite{Dun} independently, is an example of such a resemblance. Let $\Phi:\ZR^+\to\ZR^+$ be non-decreasing function and $(X,\ZB,\mu)$ be a probability space.  Denote by $L_\Phi (X)$ the class of $\ZB$-measurable functions $f$ on $X$ with $\Phi(|f|)\in L^1(\ZT)$.  The class $L_\Phi(X)$ corresponding to a function 
\begin{equation}\label{a51}
	\Phi(t)=t\left(1+\max\{0,\log^n t\}\right),\quad n\ge 1,
\end{equation} 
will be denoted by $L\log^{n} L(X)$. Clearly this class of function is strongly included in $L^1(X)$.
\begin{OldTheorem}[Dunford-Zygmund]\label{OT1}
	Let $U_1,\ldots,U_n$ be measure-preserving one-to-one transformations of a probability space $(X,\ZB,\mu)$. Then for any function $f\in L\log^{n-1} L(X)$ the averages 
	\begin{equation}\label{a25}
		\frac{1}{s_1\cdots s_{n}}\sum_{j_1=0}^{s_1-1}\cdots\sum_{j_n=0}^{s_n-1}f\left(U_1^{j_1}\cdots U_n^{j_n} x	\right)
	\end{equation}
converge a.e. as $\min_j s_j\to\infty $.
\end{OldTheorem}
This result has been generalized for general contraction operators on $L^1$, considering those instead of the operators $f\to f\circ U_k$ generated by the measure-preserving transformations  $U_k$ (Dunford-Schwartz \cite{DuSw}, Fava \cite{Fava}). Hagelstein and Stokolos in \cite{Sto} proved the sharpness of the class of  functions $L\log^{n-1} L(X)$ in the context of \trm{OT1}.  Namely,
\begin{OldTheorem}[Hagelstein-Stokolos]\label{OT2}
Let a collection of invertible commuting measure-preserving transformations $\ZU=\{U_k:\,k=1,2,\ldots, n\}$ be non-periodic, that is for any non-trivial collection of integers $p_k\in \ZZ$, $k=1,2,\ldots, n$ we have
\begin{equation}
	\mu\{U_1^{p_1}\circ\ldots\circ U_n^{p_n}(x)=x\}=0.
\end{equation}
If $\Phi(t)=o(t\log^{n-1}t)$ as $t\to\infty$, then there exists a function $f\in L_\Phi(X)$ such that averages \e{a25} unboundedly diverge a.e..
\end{OldTheorem}

\begin{definition}
	A set of invertible commuting measure-preserving (ICMP) transformations $\ZU=\{U_k:\,k=1,2,\ldots, n\}$ is said to be \emph{dependent} if there is a non-trivial collection of integers $p_k\in \ZZ$, $k=1,2,\ldots, n$, such that 
	\begin{equation}\label{a21}
		(U_1^{p_1}\circ\ldots\circ U_n^{p_n})(x)=x
	\end{equation}
almost everywhere on $X$. If there is no such a collection of integers $p_k$, then we say $\ZU$ is \emph{independent}. 
	The \emph{rank of $\ZU$} denoted by $\rank(\ZU)$ will be called the largest integer $r$ for which there is an independent subset of cardinality $r$ in $\ZU$.
\end{definition}
\begin{remark}
	Note that according to our definition, the independence of $\ZU$ requires the failure of \e{a21} on a set of positive measure for any non-trivial collection of integers $\{p_k\}$, while the condition of non-periodicity in \trm{OT2} is a stronger version of independence, since in this case the failure of \e{a21} is required almost everywhere.
\end{remark}

The main result of the present paper provides a generalization of \trm{OT1}. Namely, it says that in fact a.e. convergence of averages \e{a25} holds in a larger class of functions $L\log^{d-1} L\supset L\log^{n-1} L$, where  $d=\rank(\ZU)\le n$. First we prove the following weak type maximal inequality, where $\Log_n t$ denotes the function in \e{a51}, i.e.
\begin{equation}
	\Log_n(t)= t\left(1+\max\{0,\log^n t\}\right).
\end{equation}
\begin{theorem}\label{T2}
	Let $\ZU=\{U_k:\,k=1,2,\ldots, n\}$ be a set of ICMP transformations of rank $d$.
	Then, for any function  $f\in L\log^{d-1}L(X)$  and $\lambda >0$, we have 
	\begin{align}
		\mu\Biggl\{x\in X:\, \sup_{s_j\ge 0}\frac{1}{s_1\ldots s_n}\sum_{k_1=0}^{s_1-1}\cdots \sum_{k_n=0}^{s_n-1}&\left|f\left((U_1^{k_1}\circ\cdots \circ U_n^{k_n})(x)\right)\right|>\lambda\Biggr\} \\
		\le &C(\ZU)\int_X\Log_{d-1}\left(\frac{|f|}{\lambda}\right),\label{a22}
	\end{align}
 where $C(\ZU)$ is a constant depending only on $\ZU$.
\end{theorem}
As a  corollary of \e{a22} we obtain the following.
\begin{theorem}\label{T3}
	Let $\ZU=\{U_k:\,k=1,2,\ldots, n\}$ be a set of ICMP transformations of rank $d$.
	Then, for any function  $f\in L\log^{d-1}L(X)$ the averages \e{a25}	converge almost everywhere as $\min s_k\to\infty $.
\end{theorem}
\begin{remark}
	We will see in the last section that the class $L\log^{d-1}L(X)$ of the functions in \trm{T3} is optimal. More precisely, if the corresponding independent subset of cardinality $d=\rank(\ZU)$ in $\ZU$ is ``strongly independent" (i.e. non-periodic), then  under the condition $\Phi(t)=o(t\log^{d-1}t)$ there exists a function $f\in L_\Phi(X)$ with a.e. diverging averages \e{a25}.
\end{remark}
The inequality \e{a22} will be deduced from a maximal inequality on $\ZR^n$. Let $A:\ZR^n\to \ZR^d$ be a linear operator given by the matrix
\begin{equation}\label{a20}
	A=\{a_{kj}:\, 1\le j\le n,\, 1\le k\le d\}
\end{equation}
of size $d\times n$ ($d$-rows and $n$-columns). We consider the maximal function
\begin{equation}\label{a18}
	M_Af(\x)=\sup_{R}\frac{1}{|R|}\int_{R}|f(\x+A\cdot \t)|d\t,\quad \x\in \ZR^d,
\end{equation}
where $\sup $ is taken over all $n$-dimensional symmetric intervals 
\begin{equation}
	R=\left\{\t=(t_1,\ldots,t_n)\in \ZR^n:\,t_j\in [-r_j,r_j],\,j=1,2,\ldots,n\right\}\subset \ZR^n.
\end{equation}
Denote by $\rank A$ the rank of the matrix $A$. 
\begin{theorem}\label{T1}
	Let $A$ be the matrix \e{a20} and $r=\rank A$. Then for any function $f\in L(\log^+ L)^{r-1}(\ZR^d)$ it holds the bound
	\begin{equation}\label{a16}
		|\{\x\in \ZR^d:\, M_A f(\x)>\lambda)\}|\le C(A)\int_{\ZR^d}\Log_{r-1}\left(\frac{|f|}{\lambda}\right),
	\end{equation}
	where $C(A)$ is a constant, depending only the matrix $A$.
\end{theorem}
\begin{remark}
	Observe that if $n=d=r$ and $A$ is the identity matrix of size $n$, then \e{a18} gives the well-known strong maximal function on $\ZR^n$, correspondingly, \e{a16} becomes the weak type inequality due to M.~de~Guzman \cite{Guz1} (see also \cite{Guz2}).
	 Moreover, inequality \e{a16} holds even if $A$ is a general invertible matrix and it follows from Guzman's inequality of \cite{Guz1}, simply using the equivalence of rectangular and parallelepiped differentiation bases on $\ZR^n$. Our proof of the full version of inequality \e{a16} is a reduction of the general case to the case of invertible $A$.
	\begin{remark}
		Note that papers \cite{Dun} and \cite{Zyg} suggest different proofs of \trm{OT1}. The proof of \cite{Dun} is straightforward and the convergence of averages \e{a25} was established only for the functions in $L^p$, $1<p<\infty$, while Zygmund \cite{Zyg} provides an inequality, which is the analogue of a similar inequality for the strong maximal function, originally proved in \cite{JMZ}. The latter is the weaker version of Guzman's inequality of \cite{Guz1} .  
	\end{remark}
\end{remark}

\section{Proof of \trm{T1}}

We will use the following equivalent form of the maximal function \e{a18}
\begin{equation}\label{a14}
M_\ZU f(\x)=\sup_{r_k>0}\frac{1}{2^nr_1\cdots r_n}\int_{-r_1}^{r_1}\cdots \int_{-r_n}^{r_n}|f(\x+t_1\u_1+\cdots+t_n\u_n)|dt_1\cdots dt_n,
\end{equation}
where the vector set $\ZU=\{\u_k,\,k=1,2,\ldots, n\}$ is formed by the columns of the matrix \e{a20}. So the rank of vectors $\ZU$ coincides with the rank of the matrix $A$.
Once again note that that if the collection of vectors are independent, i.e. the matrix $A$ is invertible, then inequality \e{a16} is known, and we are going to reduce the general case to the case of invertible $A$. We need several lemmas, concerning parallelepipeds in $\ZR^d$ and associated measures. 

For a vector $\x=(x_1,\ldots,x_d)\in \ZR^d$ we denote $|x|=(x_1^2+\ldots+x_d^2)^{1/2}$. Given set of vectors $\ZV\subset \ZR^d$ we denote by $\sp(\ZV)$ the linear space generated by $\ZV$ (sometimes this Euclidean space will be denoted by $\ZR_\ZU$).
The notation $|E|$ will stand for the Lebesgue measure of a set $E$ in an Euclidean space.
\begin{definition}
	Let $\ZU=\{\u_k:\,k=1,2,\ldots,n\}\subset \ZR^d$ be a set of unit vectors. Call parallelepiped in $ \ZR^d$ a set of the form 
	\begin{equation}\label{a9}
		R=\left\{\x\in \ZR^d:\,\x=t_1\u_1+\ldots+t_n\u_n,\, t_j\in [-r_j,r_j] \right\}.
	\end{equation} 
The family of all parallelepipeds \e{a9} generated by a fixed set of vectors $\ZU$ will be denoted by $\ZP_\ZU$.
\end{definition}
Note that parallepipeds can have different representations \e{a9}. Clearly the arithmetic sum of two parallelepipeds $R$, $Q$
\begin{equation*}
	R+Q=\{\x+\t:\,\x\in R,\,t\in Q\}
\end{equation*}
 is again a parallelepiped. For two parallelepipeds $R$ and $Q$ we write $Q\prec R$ if there is a parallelepiped $R'$ such that $Q=R+R'$. 
\begin{lemma}\label{L1}
	If $\ZU=\{\u_k:\,k=1,2,\ldots,n\}$ is a basis set of vectors in $\ZR^n$ and $R\in \ZP_\ZU$ has a representation \e{a9}, then 
	\begin{equation}\label{a12}
		\{\x\in \ZR^n:\, |\x|\le 1\}\subset \frac{C(\ZU)}{\min_jr_j}\cdot R,
	\end{equation}
where $C(\ZU)$ is a constant, depending only on the set of vectors $\ZU$.
\end{lemma}
\begin{proof}
	For any $j=1,2,\ldots,n$ we consider hyperplanes $\Gamma_j^+$ and $\Gamma_j^-$ in $\ZR^n$ defined 
	\begin{align*}
		\Gamma_j^\pm=\{\x=t_1\u_1+\ldots+t_n\u_n:\, t_j=\pm r_j,\, t_i\in \ZR,\, i\neq j\}
	\end{align*} 
and let $S_j$ be the closed strip domain lying between the hyperplanes $\Gamma_j^\pm$. We have $R=\cap_jS_j$. Denote by $h_j$ the distance of the hyperplanes $\Gamma_j^+$ and $\Gamma_j^-$ from the origin. It is clear that
\begin{equation}\label{a13}
	\{\x\in \ZR^n:\, |\x|\le \min_jh_j\}\subset R.
\end{equation}
One can also check that $c_j=h_j/r_j$ are constants, depending only on $\ZU$. Denote $C(\ZU)=(\min_jc_j)^{-1}$. From \e{a13} we obtain
\begin{equation*}
	\{\x\in \ZR^n:\, |\x|\le 1\}\subset \frac{1}{\min_jh_j}\cdot R\subset   \frac{C(\ZU)}{\min_jr_j}\cdot R
\end{equation*}
and so \e{a12}.
\end{proof}
A version of the following lemma in the case of $d=2$ was proved by Guzm\'an-Welland in \cite{GuWe} (see also \cite{Guz2}, chap. 6, Lemma 2.1). 
\begin{lemma}[Guzm\'an-Welland]\label{L2}
	Let $\ZU=\{\u_k:\, k=1,2,\ldots,n\}$ be a set of unit vectors in $\ZR^d$. Then for any parallelepiped $R\in \ZP_\ZU$ there exist a subset $\ZV\subset \ZU$ of independent vectors and a parallelepiped $Q\in \ZP_\ZV$ such that 
	\begin{align}
	&\rank(\ZV)=\rank(\ZU),\label{a26}\\
	&Q\prec R,\label{a27}\\
	&R\subset C(\ZU)\cdot Q,\label{a28}
	\end{align}
where $C(\ZU)$ is a constant depending only on the set of vectors $\ZU$.
\end{lemma}
\begin{proof}
	Suppose that $R\in \ZP_\ZU$ is the parallelepiped \e{a9}. Without loss of generality we can suppose that 
	\begin{equation}\label{a8}
		r_1\ge r_2\ge \ldots\ge r_n.
	\end{equation}
Denote
\begin{equation*}
	\ZV=\{\u_k:\,\u_k\not\in \sp\{\u_1,\ldots,\u_{k-1}\}\}\subset \ZU.
\end{equation*}
One can easily check that the vectors of  $\ZV$ are independent and $\rank(\ZV)=\rank(\ZU)$. One can split the set of vectors $\ZU$ into groups 
	\begin{align}
		&\ZU_j=\{\u_k:\, k\in (k_{j-1},k_{j}]\}, \quad j=1,2,\ldots,s,\\
		&0=k_0<k_1<\ldots<k_s=n,
	\end{align}
	such that
	\begin{align*}
		\ZV=\bigcup_{i\ge 0}\ZU_{2i+1},\quad \ZU_{2j}\subset \sp\left(\bigcup_{i=1}^{j}\ZU_{2i-1}\right).
	\end{align*}
Considering the parallelepipeds
\begin{equation*}
	R_j=\left\{\x\in \ZR^d:\,\x=\sum_{k=k_{j-1}+1}^{k_j}t_k\u_k,\, t_k\in [-r_k,r_k] \right\}\in \ZP_{\ZU_j},
\end{equation*}
we can write
\begin{equation*}
R=R_1+R_2+\ldots+R_s.
\end{equation*} 
Then the parallelepiped 
\begin{equation*}
	Q=\sum_{j:\,2j-1\le s}R_{2j-1}
\end{equation*}
satisfies \e{a26} and \e{a27}. If $\x\in R_{2j}$, then 
\begin{equation}\label{a50}	
	|\x|\le \sum_{i=k_{2j-1}+1}^{k_{2j}} r_j\le nr_{k_{2j-1}}.
\end{equation}
Let $Y_j$ be the subspace of $\ZR^d$ generated by the independent vectors $\cup_{i\le j}\ZU_{2i-1}$.  One can check
\begin{equation*}
	R_i\subset Y_j,\quad i=1,2,\ldots,2j.
\end{equation*}
Thus, applying \lem{L1} for the space $Y_j$, as well as \e{a8}, \e{a50}, we conclude
\begin{align}
	\frac{1}{nr_{k_{2j-1}}}R_{2j}\subset \{\x\in Y_j:\, |\x|\le 1\}&\subset \frac{C(\ZU)}{r_{k_{2j-1}}}(R_1+R_3+\ldots+R_{2j-1})\\
	&\subset \frac{C(\ZU)}{r_{k_{2j-1}}}\cdot Q
\end{align}
Thus we get $R_{2j}\subset nC(\ZU)\cdot Q$ and therefore
\begin{equation*}
	R\subset n^2C(\ZU)Q.
\end{equation*}
This gives us \e{a28}, completing the proof of lemma.
\end{proof}

Given set of unit vectors $\ZU=\{\u_k:\, k=1,2,\ldots,n\}\subset \ZR^d$. Let $\ZR_\ZU$ the subspace of $\ZR^d$ generated by the vectors $\ZU$. We associate with a parallelepiped \e{a9} a probability measure $\mu_R$ supported on $R$ as follows. First, for each $j$ we consider a probability measure $\mu_j$ uniformly distributed on the one dimensional parallelepiped $\{t\u_j:\, t\in [-r_j,r_j]\}$. The  convolution of singular measures $\mu_{j}$ is the measure $\mu_R$ defined on the Lebesgue measurable sets of $E\subset \ZR_\ZU$ by
\begin{equation}\label{a15}
	\mu_R(E)=\int_{\ZR_\ZU}\ldots \int_{\ZR_\ZU}\ZI_{E}(\v_1+\ldots+\v_n)d\mu_1(\v_1)\ldots d\mu_n(\v_n).
\end{equation}
One can check that $\mu_R$ is well-defined for any Lebesgue measurable set $E\subset \ZR_\ZU$. Denote by $f_R$ the density function of measure $\mu_R$ with respect to the Lebesgue measure on $\ZR_\ZU$. Observe that if $\ZU$ is independent, then 
\begin{align}
	f_R(\x)=\left\{\begin{array}{rrl}
		|R|^{-1}&\text { if }&\x\in R,\\
		0&\text{ if }&\x\in \ZR_\ZU\setminus  R.\label{a29}
	\end{array}
	\right.
\end{align}
\begin{lemma}\label{L3}
	Let $\ZU\subset \ZR^d$ be a set of arbitrary unit vectors and $R\in \ZP_\ZU$. Then there exists a set of independent vectors $\ZV\subset \ZU$ such that $\rank(\ZV)=\rank (\ZU)$ and there is a parallelepiped $R'\in \ZP_\ZV$ such that
	\begin{equation}\label{a10}
		\mu_R\le C(\ZU)\cdot \mu_{R'}.
	\end{equation}
\end{lemma}
\begin{proof}
	 Applying \lem{L2} in the Euclidean space $\ZR_\ZU$, we find a set of independent vectors $\ZV\subset \ZU$, $\rank(\ZV)=\rank (\ZU)$ and a parallelepiped $Q\in \ZP_\ZV$ satisfying the conditions of lemma.  
	 Since $Q\prec R$, we have $R=Q+H$ for some parallelepiped $H$ in $\ZR_\ZU$. We can write
	\begin{align}
		\mu_R(E)&=\int_{\ZR_\ZU}\int_{\ZR_\ZU}\ZI_{E}(\v+\v')d\mu_Q(\v)d\mu_H(\v')\\
		&=\int_{\ZR_\ZU}\int_{\ZR_\ZU}\ZI_{E}(\v+\v')f_Q(\v)d\v d\mu_H(\v')\\
		&=\frac{1}{|Q|}\int_{\ZR_\ZU}\int_{\ZR_\ZU}\ZI_{E}(\v+\v')\ZI_Q(\v)d\v d\mu_H(\v')\\
		&\le \frac{|E|}{|Q|}.
	\end{align} 
	This clearly  implies 
	\begin{equation}\label{a32}
		\|f_R\|_\infty \le \|f_Q\|_\infty= |Q|^{-1}.
	\end{equation} 
	Denote $R'=C(\ZU)Q$, where $C(\ZU)$ is the constant in \e{a28}. From \e{a28} and \e{a29} we have
	\begin{align}
		&R\subset R',\label{a30}\\
		&\|f_{R'}\|_\infty=|R'|^{-1}=\left(C(\ZU)|Q|\right)^{-1}.\label{a31}
	\end{align}
	Combining \e{a32}  and \e{a31} we get the pointwise bound $	f_R\le C(\ZU)f_{R'}$, which implies \e{a10}.
\end{proof}
\begin{proof}[Proof of \trm{T1}]
Observe that the integral in \e{a14} may be written as a convolution of measure \e{a15} with the function $f$. Namely, we have
\begin{align}
	\frac{1}{2^nr_1\ldots r_n}\int_{-r_1}^{r_1}\ldots \int_{-r_n}^{r_n}&|f(\x+t_1\u_1+\ldots+t_n\u_n)|dt_1\ldots dt_n\\
	&=\int_{\ZR^d}|f(\x+\v)|d\mu_R(\v).\label{a17}
\end{align}
Applying \lem{L3}, for any parallelepiped $R\in \ZP_\ZU$  we find an independent vector set $\ZV\subset \ZU$ with $\rank(\ZV)=\rank (\ZU)$ and a parallelepiped $R'\in \ZP_\ZV$ such that \e{a10} holds. Thus the last integral in \e{a17} may be estimated as follows:
\begin{equation}
	\int_{\ZR^d}f(\x+\v)d\mu_R(\v)\le C(\ZU)\int_{\ZR^d}f(\x+\v)d\mu_{R'}(\v)\le C(\ZU)M_\ZV f(\x).
\end{equation}
This implies
\begin{equation*}
M_\ZU f(\x)\le C(\ZU)\max_\ZV M_\ZV f(\x),
\end{equation*}
where the maximum is taken over all the subsets $\ZV\subset \ZU$ of independent vectors such that $\rank(\ZV)=\rank(\ZU)$. For each such $\ZV$ the operator $M_\ZV$ satisfies the bound \e{a16} and the number of all collections $\ZV$ is constant, depending only on $n$ and so on $\ZU$. Thus we get \e{a16}.
\end{proof}
\section{A discrete maximal inequality}
We will need a discrete version of inequality \e{a16}. Let $\phi:\ZZ^d\to \ZR$ be a $d$-dimensional sequence and let $A=\{a_{kj}:\, 1\le j\le n,\, 1\le k\le d\}$ be an integer matrix. Consider the maximal operator
\begin{align}
	\ZD_A\phi(\n)&=\sup_{s_j\in \ZN}\frac{1}{s_1\ldots s_n}\sum_{k_1=0}^{s_1-1}\ldots\sum_{k_n=0}^{s_n-1}\phi(\n+A\cdot \k)\\
	&=\sup_{s_j\in \ZN}\frac{1}{s_1\ldots s_n}\sum_{\k=0}^{\s-1}\phi(\n+A\cdot \k),\quad \n\in \ZN^d.
\end{align}
From \trm{T1} we easily obtain the following.
\begin{corollary}\label{C2}
	For any integer matrix $A$ of $\rank(A)=r$ it holds the bound
	\begin{equation}
		\#\{\n\in \ZZ^d:\,  \ZD_A\phi(\n)>\lambda\}\le C(A)\sum_{\n\in \ZZ^d}\Log_{r-1}\left(\frac{|\phi(\n)|}{\lambda}\right).
	\end{equation}
\end{corollary}
\begin{proof}
Given multiple sequence $\phi(\m)$ consider the function
\begin{equation}\label{a52}
	f(\x)=	\sum_{\varepsilon_j=0,1,-1}\phi(m_1+\varepsilon_1,\ldots,m_n+\varepsilon_n),\text { if } [\x]=\m,\quad \m\in \ZZ_d,
\end{equation}
on $\ZR^d$, where $[\x]=([x_1],\ldots,[x_d])$ denotes the coordinate wise integer part of the vector $\x=(x_1,\ldots,x_d)$.  Clearly there is a constant $\delta=\delta(A)<1$ such that 
\begin{equation}\label{a53}
	A\left(\Delta\right)\subset (-1,1)^d,\text{ where } \Delta=[0,\delta)^n,
\end{equation}
Using \e{a52}, \e{a53}, one can check that
\begin{equation}
	\phi(\n+A\cdot \k)\le f(\x+A\cdot \t)\text { if } \t\in \k+\Delta,\, [\x]=\n.
\end{equation}
Thus we obtain
\begin{align}
	\sum_{\k=0}^{\s-1}\phi(\n+A\cdot \k)&\le \sum_{\k=0}^{\s-1}\frac{1}{|\Delta|}\int_{\k+\Delta}|f(\x+A\cdot \t)|d\t\\
	&\le \frac{1}{|\Delta|}\int_R|f(\x+A\cdot \t)|d\t,
\end{align}
for any $\x$ with $[\x]=\n$, where
\begin{equation}
	R=\{\t\in \ZR^n:\,t_j\in [-1,s_j] ,\, j=1,\ldots,n\}.
\end{equation}
This implies
\begin{equation}
	\ZD_A\phi(\n)\le C(A)M_Af(\x)\text { if } [\x]=\n\in \ZZ^d
\end{equation}
and so
\begin{align}
		\#\{\n\in \ZZ_+^d:\,  \ZD_A\phi(\n)>\lambda\}&\le 	|\{\x\in \ZR^d:\,  M_Af(\x)>\lambda/C(A)\}|\\
		&\le C(A)\int_{\ZR^d}\Log_{r-1}\left(\frac{|f|}{\lambda}\right)\\
		&\le C(A)\sum_{\n\in \ZZ^d}\Log_{r-1}\left(\frac{|\phi(\n)|}{\lambda}\right).
\end{align}
This completes the proof.
\end{proof}
\section{Proofs of Theorems \ref{T2} and \ref{T3}}

\begin{proof}[Proof of \trm{T2}]
	Since $\rank(\ZU)=d$, without loss of generality we can suppose that $U_1,\ldots,U_d$ are independent and 
	\begin{equation}\label{a62}
		U_k^{l_k}=U_1^{a_{1,k}}\circ\cdots\circ U_d^{a_{d,k}},\quad d<k\le n,
	\end{equation}
where $l_k\ge 1$ and $a_{j,k}$ are some integers. First we suppose that $l_k=1$. Thus we can write
\begin{align}
	&f\left((U_1^{k_1}\circ\cdots \circ U_n^{k_n})(x)\right)\\
	&\qquad\qquad=f\left((U_1^{k_1+a_{1,d+1}k_{d+1}+\cdots+a_{1,n}k_n}\circ\cdots \circ U_d^{k_d+a_{d,d+1}k_{d+1}+\cdots+a_{d,n}k_n})(x)\right)\\
	&\qquad\qquad=\phi(x,A\cdot \k),\label{a55}
\end{align}
where
\begin{align}
	&\phi(x,\n)=f\left((U_1^{n_1}\circ\cdots \circ U_d^{n_d})(x)\right),\\
	&\n=(n_1,\ldots,n_d)\in \ZZ^d,
\end{align}
and
\begin{equation*}
	A=\begin{pmatrix}
		1&0&\cdots&0&a_{1,d+1}&\cdots &a_{1,n}\\
		0&1&\cdots&0&a_{2,d+1}&\cdots &a_{2,n}\\
		\cdot &\cdot &\cdots&\cdot&\cdot &\ldots &\cdot\\
		0&0&\ldots&1&a_{d,d+1}&\ldots &a_{d,n}
	\end{pmatrix}
\end{equation*}
is a matrix of size $d\times n$. Let
\begin{equation}\label{a57}
f_M^*(x,\n)=\max_{1\le s_j\le M}\frac{1}{s_1\cdots s_n}\sum_{\k=0}^{\s-1}|\phi(x,\n+A\cdot \k)|,
\end{equation}
where $M\in \ZN$ and denote
\begin{align}
	&E_\lambda(x)=\{\n:\, 1\le n_j\le N:\, f_M^*(x,\n)>\lambda\},\\
	&E_\lambda (\n)=\{x:\, f_M^*(x,\n)>\lambda\},\quad \n\in \ZZ^d,\\
	&E_\lambda=\{(x,\n):\, 1\le n_j\le N,\,f_M^*(x,\n)>\lambda\}=\cup_{x\in X}E_\lambda(x)\\
	&\qquad\qquad\qquad\qquad\qquad\qquad\qquad\qquad\quad\quad\,=\cup_{1\le n_j\le N}E_\lambda(\n).\label{a23}
\end{align}
Taking into account \e{a55}, observe that inequality \e{a22} is the same as 
\begin{equation}\label{a54}
\lim_{M\to\infty}\mu(E_\lambda (\textbf{0}))\le C(U)\int_X\Log_{d-1}\left(\frac{|f|}{\lambda}\right).
\end{equation}
In \e{a57} the coordinates of $A\cdot \k$ may vary in the interval $[-R, R]$, where $R=R(A,M)$ is a constant depending only on the matrix $A$ and the integer $M$. From \cor{C2} it follows that
\begin{equation}
\#(E_\lambda (x))\le C(A)\sum_{1\le n_j\le N+R}\Log_{r-1}\left(\frac{|\phi(x,\n)|}{\lambda}\right)\text{ for all } x\in X.
\end{equation}
Then, since $U_k$ are measure-preserving, the sets $E_\lambda (\n)$ have equal measures for different $\n\in\ZZ^d$. Thus from \e{a23} we obtain
\begin{align}
	\mu(E_\lambda (\textbf{0}))&=\frac{1}{N^d}\sum_{1\le n_j\le N}\mu(E_\lambda (\textbf{\n}))=\frac{1}{N^d}\int_X\#\left(E_\lambda(x)\right)\\
	&\le \frac{C(A)}{N^d}\sum_{1\le n_j\le  N+R}\int_X\Log_{r-1}\left(\frac{|\phi(x,\n)|}{\lambda}\right)\\
	&=\frac{C(A)(N+R)^d}{N^d}\int_X\Log_{r-1}\left(\frac{|f|}{\lambda}\right).
\end{align}
Fixing $M$ and letting $N\to\infty$, we get
\begin{equation*}
	|E_\lambda (\textbf{0})|\le C(A)\int_X\Log_{r-1}\left(\frac{|f|}{\lambda}\right),
\end{equation*}
which implies \e{a54}. The general case $l_k\ge 1$ can be easily deduced from the case of $l_k=1$. Fix an integer vector $\zr=(r_{d+1},\ldots,r_n)$, $0\le r_j<l_j$, and denote by $Q_{s_1,\ldots,s_d}^\zr f(x)$ the sum of functions
\begin{equation}\label{a58}
\left|f\left(U_1^{k_1}\ldots U_n^{k_n} x\right)\right|,
\end{equation}
over the integer vectors $\k=(k_1,\ldots,k_n)$, satisfying 
\begin{align}
&1\le k_j<s_j,\quad 1<j\le n,\label{a60}\\
& k_j= \bar k_jl_j+r_j,\quad \bar k_j\in \ZZ,\quad  d<j\le n.\label{a59}
\end{align}
Under the conditions \e{a59} we can write
\begin{equation}\label{a61}
f\left(U_1^{k_1}\ldots U_n^{k_n} x	\right)=\bar f\left(U_1^{ k_1}\ldots U_d^{ k_d}\bar U_{d+1}^{\bar k_{d+1}}\ldots \bar U_{n}^{\bar k_{n}}x	\right)
\end{equation}
where
\begin{align*}
&\bar f(x)=f\left(U_{d+1}^{r_{d+1}}\ldots U_n^{r_n} x\right),\\
&\bar U_j= U_j^{l_j},\quad d<j\le n.
\end{align*}
From \e{a62} it follows that
\begin{equation}\label{a63}
\bar U_k=U_1^{a_{1,k}}\circ\ldots\circ U_d^{a_{d,k}},\quad d<k\le n,
\end{equation}
Denote by $\alpha(\s,\zr)$ the number of integer vectors $\k=(k_1,\ldots,k_n)$, satisfying \e{a60} and \e{a59}. 
According to \e{a61} and \e{a63} we can say that 
\begin{equation}\label{a64}
	\frac{Q_{s_1,\ldots,s_n}^\zr  f(x)}{\alpha(\s,\zr)}
\end{equation}
are certain ergodic averages, obeying the case of $l_k=1$ in \e{a62}. Thus we conclude that the averages \e{a64} satisfy weak estimate \e{a22} for all vectors $\zr$.
On the other hand, taking into account $\alpha(\s,\zr)\le s_1\ldots s_n$, we have
\begin{align*}
\frac{1}{s_1\ldots s_n}\sum_{k_1=0}^{s_1-1}\ldots &\sum_{k_n=0}^{s_n-1}\left|f\left((U_1^{k_1}\circ\ldots \circ U_n^{k_n})(x)\right)\right|\\
&=\frac{1}{s_1\ldots s_n}\sum_\zr Q_{s_1,\ldots,s_n}^\zr f(x)\\
&=\sum_\zr \frac{\alpha(\s,\zr)}{s_1\ldots s_n}\frac{Q_{s_1,\ldots,s_n}^\zr  f(x)}{\alpha(\s,\zr)}\\
&\le \sum_\zr\frac{Q_{s_1,\ldots,s_n}^\zr  f(x)}{\alpha(\s,\zr)}.
\end{align*}
Thus, since the averages \e{a64} satisfy the weak estimate \e{a22} and the number of different vectors $\zr=l_{d+1}\ldots l_n$ is a constant depending on $\ZU$ only, we obtain \e{a22} in full generality. Theorem is proved.
\end{proof}

	\begin{proof}[Proof of \trm{T3}]
		According to \trm{OT1} the averages \e{a25} converge a.e. for any function from $L\log^{n-1}L$ and so for any $f\in L^\infty(X)$.  	To prove convergence for any $f\in L\log^{d-1} L(\ZT)$, fix $\varepsilon > 0$ and choose a function $g \in L^{\infty}$ such that 
\begin{equation}
		\int_X\Log_{d-1}\left(\frac{|f-g|}{\varepsilon }\right)< \varepsilon.
\end{equation}
		Applying \e{a22}, for the averages  
		\begin{equation*}
		\mathrm{A}_{\m}(f)=	\frac{1}{m_1\ldots m_{n}}\sum_{j_1=0}^{m_1-1}\ldots\sum_{j_n=0}^{m_n-1}f\left(U_1^{j_1}\ldots U_n^{j_n} x	\right)
		\end{equation*}
we obtain
		\begin{align}
			&\mu\left\{x : \limsup _{\min n_j,\,\min m_j\rightarrow \infty} |	\mathrm{A}_{\n}(f)-\mathrm{A}_{\m}(f)| >2 \varepsilon\right\} \\
			&=\mu\left\{x : \limsup _{\min n_j,\,\min m_j\rightarrow \infty} |	\mathrm{A}_{\n}(f-g)-\mathrm{A}_{\m}(f-g)| >  2\varepsilon\right\} \\
			&\le \mu\left\{x : \sup _{\n }| \mathrm{A}_{\n}\left(f-g\right) \mid> \varepsilon\right\} \\
			&\le C(\ZU)\int_X\Log_{d-1}\left(\frac{|f-g|}{\varepsilon }\right)< C(\ZU)\varepsilon.
		\end{align}
		This implies a.e. convergence of $A_\n(f)$, completing the proof of the theorem.
	\end{proof}
\section{Sharpness in \trm{T3} and an extension}	
Let us show that the class of functions $L\log^{d-1}L(X)$ in \trm{T3} is optimal. So suppose the rank of $\ZU=\{U_k:\,k=1,2,\ldots, n\}$ is $d$ and $\{U_1,\ldots, U_d\}$ is the corresponding independent subset $\ZU$, which is moreover non-periodic. According to \trm{OT2} for $\Phi(t)=o(t\log^{d-1}t)$ there exists a function $f\in L_\Phi(X)$ with a.e. diverging averages 
	\begin{equation}\label{a69}
	\frac{1}{s_1\ldots s_{d}}\sum_{j_1=0}^{s_1-1}\ldots\sum_{j_d=0}^{s_d-1}f\left(U_1^{j_1}\ldots U_d^{j_d} x	\right)
\end{equation}
It turns out that for the same function $f$ we have a.e. divergence of the averages 
	\begin{equation}\label{a70}
	\frac{1}{s_1\ldots s_{n}}\sum_{j_1=0}^{s_1-1}\ldots\sum_{j_n=0}^{s_n-1}f\left(U_1^{j_1}\ldots U_n^{j_n} x	\right).
\end{equation}
This immediately follows from the following lemma.
\begin{lemma}
	Let $\ZU=\{U_k:\,k=1,2,\ldots, n\}$ be a set of measure-preserving transformations and $d\le n$.  If averages \e{a69}
diverge unboundedly a.e, then  extended averages \e{a70} also diverge unboundedly a.e. 
\end{lemma}
\begin{proof}
	Denote by $A_{\s}(f)$ and $ \bar A_{\s}(f)$ the averages \e{a69}  and \e{a70} respectively and consider the functions
	\begin{equation}
		M_{p}(f)=\max_{\s\in \ZZ_+^d,\, s_j\ge p}A_\s (f),\quad \bar M_{p}(f)=\max_{\s\in \ZZ_+^n,\, s_j\ge p}\bar A_\s (f).
	\end{equation}
	The unbounded divergence  of averages \e{a69} implies $M_p(f)=\infty$ a.e. for any $p>0$. If $\s=(s_1,\ldots,s_d)$ and $\bar \s=(s_1,\ldots, s_d,\ldots,s_n)$, then we have
	\begin{equation*}
	\bar A_{\bar \s} (f)\ge \frac{ A_{ \s} (f)}{s_{d+1}\ldots s_n},
	\end{equation*}
and thus, for any $p>0$
\begin{equation}
	\bar M_p(f)\ge \frac{1}{p^{n-d}}M_p(f)=\infty \text { a.e.}.
\end{equation}
\end{proof}

A set of real numbers 
\begin{equation}\label{Th}
	\Theta=\{\theta_1,\theta_2,\ldots,\theta_n\}
\end{equation}
is said to be \emph{dependent} (with respect to the rational numbers) if there are is a  non-trivial collection of integers $r_k$, $k=1,2,\ldots,n$, such that 
\begin{equation}
	r_1\theta_1+r_2\theta_2+\ldots+r_n\theta_n=0\mod 1,
\end{equation}
If there are no such integers, then we say that $\Theta$ is \emph{independent}.
The rank of a collection $\Theta=\{\theta_1,\theta_2,\ldots,\theta_n\}$ will be called the largest integer $d$, for which there is an independent subset of cardinality $d$ in $\ZU$. Consider the probability space of Lebesgue measure on $\ZT=\ZR/\ZZ$ with  modulo one addition. Applying \trm{T1} and the ergodicity of the rotation mapping $x\to x+\theta$ for an irrational $\theta$, we obtain

\begin{corollary}
If $\e{Th}$ is a sequence of rank $d$, then
	
	1) for any $f\in L\log^{d-1} L(\ZT)$ the limit below holds a.e. 
	\begin{equation}\label{a1}
		\lim_{\min\{s_k\}\to\infty }\frac{1}{s_1\cdots s_n}\sum_{k_1=0}^{s_1-1}\cdots\sum_{k_n=0}^{s_n-1}f(x+k_1\theta_1+\cdots+k_n\theta_n)=\int_\ZT f(x)dx,
	\end{equation}

2) for any increasing function $\Phi:\ZR^+\to\ZR^+$, satisfying $\Phi(t)=o(t(\log t)^{d-1})$, there exists a function $f\in L_\Phi(\ZT)$ such that the averages in \e{a1} are a.e. divergent as $\min\{s_k\}\to \infty$.
\end{corollary}

\end{document}